\documentclass[12pt]{amsart}
\usepackage{amssymb,amsmath, enumerate, mathrsfs}
\usepackage{stackengine,scalerel}
\usepackage[T1]{fontenc}
\usepackage{lmodern}
\usepackage{amsthm}
\usepackage{amsfonts}
\usepackage{amstext}
\usepackage{graphicx}
\usepackage{appendix}

\newtheorem{thm}{Theorem}[section]
\newtheorem{prop}[thm]{Proposition}
\newtheorem{lem}[thm]{Lemma}
\newtheorem{cor}[thm]{Corollary}
\numberwithin{equation}{section} 

\theoremstyle{definition}

\newtheorem{defn}[thm]{Definition}

\theoremstyle{remark}

\newcommand{\ato}{\alpha_{21}}

\renewcommand{\H}{\mathcal H^\ast}
\newcommand{\Kg}{K_\Gamma (w, z)}
\newcommand{\Kh}{K_h}

\newcommand{\Rh}{\mathcal R_h}

\newcommand{\s}{\mathfrak h_{z_1, z_2}}
\newcommand{\tho}{\theta_1}
\newcommand{\ttw}{\theta_2}

\newcommand{\Tg}{\mathbf{t}_\Gamma}

\newcommand\Humlaut[1]{\stackengine{-.05ex}{#1}{\hstretch{.8}{\vstretch{.65}{%
  \mkern1mu\scriptscriptstyle''}}}{O}{c}{F}{F}{S}}

\newcommand{\bndry}{b}

\renewcommand{\H}{\mathcal H}

\begin{document} 

\title[Symmetrization of a restricted Cauchy kernel] 
{Symmetrization of a Cauchy-like kernel on curves}

\author[Lanzani and Pramanik]{Loredana Lanzani
and Malabika Pramanik}
\address{
Dept. of Mathematics,       
Syracuse University 
Syracuse, NY 13244-1150 USA}
  \email{llanzani@syr.edu}
\address{
Dept. of Mathematics\\University of British Columbia
\\
Room 121, 1984 Mathematics Rd.
\\Vancouver, B.C. Canada   V6T 1Z2}
\email{malabika@math.ubc.ca}
\begin{abstract} Given a curve $\Gamma\subset \mathbb C$ with specified regularity, we
 investigate boundedness and positivity for a certain three-point symmetrization of a 
Cauchy-like kernel $K_{\Gamma}$ 
whose definition is
dictated by the geometry and complex function theory of the domains bounded by $\Gamma$. Our results show that $\mathtt S[\text{Re} K_{\Gamma}]$ and $\mathtt S[\text{Im} K_{\Gamma}]$ (namely, the symmetrizations of the real and imaginary parts of $K_{\Gamma}$)  behave very differently from their
counterparts for the 
Cauchy kernel
previously studied in the literature. 
For instance, 
the quantities $\mathtt S[\text{Re} K_{\Gamma}](\mathbf z)$ and $\mathtt S[\text{Im} K_{\Gamma}](\mathbf z)$ can behave like 
$\frac32c^2(\mathbf z)$ and $-\frac12c^2(\mathbf z)$, where $\mathbf z$ is any three-tuple of points in $\Gamma$ and
$c(\mathbf z)$ is the Menger curvature of $\mathbf z$.
For the 
original 
Cauchy kernel, an iconic result of M. Melnikov gives that the symmetrized forms of the real and imaginary parts are each equal to $\frac12c^2(\mathbf z)$ for all three-tuples in $\mathbb C$. 
\end{abstract}
\maketitle

\section{Introduction}
 Given a complex-valued function $K(w, z)$
defined on a subset of $\mathbb C^2$ except possibly the diagonal $\{(z,z) : z \in
\mathbb C\}$,  we consider the following symmetric form
associated with $K$: 

\begin{equation}\label{E:K-sym} 
\mathtt S\,[K](\mathbf z) :=   \sum\limits_{\sigma\in S_3}
  K(z_{\sigma(1)},\, z_{\sigma(2)})\, 
  \overline{K(z_{\sigma(1)},\, z_{\sigma(3)})} 
\end{equation} 
where $S_3$ is the group of permutations over three elements,  and $\mathbf z = \{z_1, z_2, z_3\}$ is any three-tuple of points in $\mathbb C$ for which the above expression is meaningful. The above definition of course also makes sense for real-valued $K(w, z)$.

A primary
reference is
 the symmetric form of the kernel
 \begin{equation}\label{E:ker-0}
K_0(w, z) \ :=\ 
\frac{1}{w-z}, \qquad z, w \in \mathbb C, \; z \ne w\, 
\end{equation}
along with
three seminal
identities 
discovered by M. Melnikov \cite{M};
these are
 \begin{align}\label{E:sym-id-0p}
\mathtt S\,[K_0](\mathbf z) &= 
c^2(\mathbf z), \\
\mathtt S\,[\mathrm{Re}K_0](\mathbf z) &= 
\frac{1}{2}\, c^2(\mathbf z), \label{E:sym-id-0pr}\\
\mathtt S\,[\mathrm{Im}K_0](\mathbf z) &=
\frac{1}{2}\, c^2(\mathbf z)\, . \label{E:sym-id-0pi}
  \end{align}
  \vskip0.1in
\noindent  Here 
   $\mathbf z = \{z_1, z_2, z_3\}$ is any three-tuple of distinct points in $\mathbb C$;  $\mathrm{Re}K_0$ (resp. $\mathrm{Im}K_0$) is the real (resp. imaginary) part of \eqref{E:ker-0}, and $c(\mathbf z)$ denotes the {\bf Menger curvature} associated with $\mathbf z$; we recall that 
$c(\mathbf z)$ is defined to be either zero or the reciprocal of the radius of the unique circle passing through $ z_1, z_2$ and $z_3$ according to whether the three points are, or are not, collinear.
Identities  \eqref{E:sym-id-0p}-\eqref{E:sym-id-0pi} are often referred to
as ``Melnikov's miracle'' in recognition of the profound 
influence they have had on the theory of singular integral operators \cite[p. 2]{T}.
\vskip0.1in
Because the kernel $K_0$ in \eqref{E:ker-0} is defined in the maximal setting of $(z, w)\in \mathbb C\times\mathbb C\setminus \{z=w\}$ we will henceforth refer to it as the {\bf universal Cauchy kernel} or the {\bf original Cauchy kernel}.
In this paper we study the effects of the symmetrization \eqref{E:K-sym} on the {\bf $\Gamma$-restricted Cauchy kernel}
\begin{equation}\label{E:Ker-Gamma}
\Kg \ :=\ 
\frac{1}{2\pi i}\,\frac{\Tg(w)}{w-z}, \qquad z, w \in \Gamma, \; z \ne w\, 
\end{equation}
where $\Gamma\subset\mathbb C$ is a given rectifiable curve, and $\Tg (w)$ is the unit tangent vector to $\Gamma$ at $w$
in the counterclockwise direction.
\vskip0.1in
At  first glance  $K_\Gamma$ would appear to be only a minor variant of $K_0$ because  $\mathtt S\,[K_\Gamma](\mathbf z) = \mathtt S\,[K_0](\mathbf z)$ for {\em any} $\Gamma$ and any three-tuple
 of distinct points.
But the presence of $\Tg(w)$ renders the function $K_\Gamma (w, z)$ non-homogeneous (unless $\Gamma$ is a line): in this respect the 
$\Gamma$-restricted Cauchy kernel \eqref {E:Ker-Gamma} is substantially different in nature from the restriction to $\Gamma$ of the
universal 
Cauchy kernel \eqref{E:ker-0}. This distinction is especially relevant in complex analysis, where the numerator $\Tg(w)$ in \eqref{E:Ker-Gamma}
 is indispensable
already for characterizing the holomorphic Hardy spaces $H^p(\Gamma, \sigma)$. Here and throughout $\sigma$ will denote the arc-length measure for $\Gamma$.

\vskip0.1in
Our long-term goal is to employ kernel-symmetrization techniques to study the complex function theory of the domains bounded by $\Gamma$. 
 As a first step in this direction, here we establish base-line results for $\mathtt S[K](\mathbf z)$ where
  $K$ is any of
$\mathrm{Re}{K_\Gamma}$ and $\mathrm{Im}{K_\Gamma}$ (the real and imaginary parts of $K_\Gamma$). 
To be precise, we consider
the following two basic features of $K_0$
   \vskip0.1in
   \begin{itemize}
   \item[{\tt(i)}] boundedness relative to  $c^2(\mathbf z)$ of  each of
    $\mathtt S\,[K_0](\mathbf z)$,
 $\mathtt S\,[\mathrm{Re}\, K_0](\mathbf z)$ and $\mathtt S\,[\mathrm{Im}K_0](\mathbf z)$; 
\vskip0.1in
 \item[{\tt(ii)}] positivity of each of   $\mathtt S\,[K_0](\mathbf z)$,
 $\mathtt S\,[\mathrm{Re}\, K_0](\mathbf z)$ and $\mathtt S\,[\mathrm{Im}K_0](\mathbf z)$
 \end{itemize}
 \vskip0.1in
which are automatically granted by \eqref{E:sym-id-0p}-\eqref{E:sym-id-0pi}, and it is these features that we explore here for $\mathtt S\,[\mathrm{Re}{K_\Gamma}](\mathbf z)$ and $\mathtt S\,[\mathrm{Im}{K_\Gamma}](\mathbf z)$.
 \subsection{The historical context: a brief review} 
  In 1995 Melnikov and Verdera \cite{MV} discovered   that 
   \eqref{E:sym-id-0p}  leads to a  new proof of the celebrated  $L^2(\Gamma, \sigma)$-regularity of the {\bf Cauchy transform} 
   for a planar Lipschitz curve 
     \cite{{Calderon},{CMM}}:
  \begin{equation}\label{E:CT}
f\mapsto   \text{p.v.}\int\limits_{\Gamma} f(w) K_0(w, z)d\sigma(w)\, .
  \end{equation}
  As noted in \cite[Section
 3.7.4]{T}, this new method of proof combined with
   \eqref{E:sym-id-0pr} 
  also shows  that $L^2$-regularity of the Cauchy transform is, in effect, {\em equivalent} to the regularity of its real part alone that is, of the 
operator obtained by replacing $K_0(w, z)$ in \eqref{E:CT} with $\mathrm{Re}K_0 (w, z)$;
  of course \eqref{E:sym-id-0pi} gives an analogous statement for $\mathrm{Im}K_0$.

 This new approach was amenable to a large class of measures $\mu$ and
  ultimately
 led to 
 ground-breaking progress towards the resolution of 
long-standing open problems in geometric measure theory known as the Vitushkin's conjecture and the Painlev\`e problem \cite{{MMV}, {T1}, {T2}}. 
At the core of this work is the discovery that Melnikov's miracle \eqref{E:sym-id-0p}-\eqref{E:sym-id-0pi}
brings to the fore deep connections between the $L^2(\mu)$-boundedness of 
 the Cauchy transform;
  the notion  of curvature of the reference measure $\mu$, \cite{M}; and the rectifiability of the support of $\mu$. This circle of ideas is often referred to as ``the curvature method for $K_0$''; we defer to the excellent surveys \cite{P} and \cite{T} for detailed reviews of the extensive literature. 

 The interplay of
analysis and geometry as manifested in \eqref{E:sym-id-0p}-\eqref{E:sym-id-0pi}, in fact already at the level of the basic features {\tt(i)} and {\tt (ii)},
has inspired the energetic pursuit of analogous connections
 for other Calder\'on-Zygmund kernels $K$ and other notions of curvature, with diverse objectives and mixed success.
 It has been observed (see for instance \cite[Section 3.7.4]{T}) that condition {\tt(ii)} fails for most kernels $K$, 
thereby ruling out
 the possibility of a non-negative curvature method valid in a broad context.
 For instance, Farag \cite{F} shows that there is no higher-dimensional analogue of Menger-like curvatures stemming from Riesz transforms with integer exponents. Prat \cite{Prat2004} on the other hand shows that $\mathtt S\,[K]$ {\em{is}} non-negative for fractional signed
Riesz kernels $K$ with homogeneity $- \alpha$, $0 < \alpha < 1$, using it to prove
unboundedness of associated Riesz transforms on certain measure spaces. Lerman and Whitehouse
\cite{{LW1},{LW2}} introduce discrete and continuous variants of Menger-type
curvatures in a real separable Hilbert space. Their definition of
curvature uses general simplices instead of three-tuples of points; curvatures such as these have applications to problems in multiscale geometry.
A number of recent articles, notably Chousionis-Mateu-Prat-Tolsa \cite{{CMPT}, {CMPT2}}; Chousionis-Prat \cite{CP}; Chunaev \cite{C}, and Chunaev-Mateu-Tolsa \cite{{CMT-1}, {CMT-2}}, explore curvature methods
for various 
kinds of singular integral operators; certain aspects of the techniques developed in those papers are relevant to the present work and are discussed in section \ref{S:appendix}. 

\subsection{Our context} While Menger curvature has
  been employed primarily to study the $L^2(\mu)$-regularity of 
Calder\`on-Zygmund operators and their implications to geometric measure theory, our long term goal is to investigate the connection between Menger curvature and 
 the {\bf Cauchy-Szeg\Humlaut{o} 
projection $\mathcal S_\Gamma$}, 
which is
the unique, orthogonal projection of $L^2(\Gamma, \sigma)$ onto $H^2(\Gamma, \sigma)$, for a given
rectifiable curve $\Gamma$.
The Cauchy-Szeg\Humlaut{o} projection is a singular integral operator whose integration kernel is almost never known in explicit form and in general it is not Calder\`on-Zygmund; what's more, $\mathcal S_\Gamma$ is trivially bounded on $L^2(\Gamma, \sigma)$ whereas proving boundedness in $L^p$ for $p\neq2$ is a difficult problem known as ``the $L^p$-regularity problem for $\mathcal S_\Gamma$''. On the other hand, the Cauchy-Szeg\Humlaut{o} projection bears an intimate connection with 
the {\bf Kerzman-Stein operator $\mathcal A_\Gamma$} whose integration kernel
is 
\begin{equation}\label{E:KS}
A_\Gamma(w, z):= \Kg - \overline{K_\Gamma (z, w)},\qquad z, w \in \Gamma, \; z \ne w.
\end{equation}
The connection with the Cauchy-Szeg\Humlaut{o} projection transpires whenever the Kerzman-Stein operator satisfies finer properties than $L^p$-regularity: for instance, compactness in $L^p(\Gamma, \sigma)$.
Such connection is one reason for
our interest in $K_\Gamma$ and its real and imaginary parts. In fact $K_\Gamma$
  is relevant to the analysis of various reproducing kernel Hilbert spaces (the holomorphic Hardy space $H^2(\Gamma, \sigma)$ being
   one such instance see e.g., \cite[p. 376]{F}), whereas 
  $\mathrm{Re}{K_\Gamma}$
  is of particular interest in potential theory since it is the integration kernel of the {\bf double layer potential operator} \cite[(1.14)]{LPg}.  We defer to \cite{{Be}, {Du}, {KS}, {LS}}
for the precise definitions and  the statements of the main results on these topics, and for references to
 the extensive literature. 

\vskip0.1in
We focus on the restricted setting of a curve parametrized as a graph
 \begin{equation}\label{E:curve}
 \Gamma = \{z= x+iA (x), \ x\in J = (a, b)\subseteq\mathbb R\}
 \end{equation}
 where the function $A(x)$ is of class $C^1$ or better
  (as specified in the statement of each result below)
 and we represent the righthand side of \eqref{E:Ker-Gamma} in parametric form, giving
 \begin{equation}\label{E:def-restricted}
K_\Gamma (w, z) = \frac{1}{2\pi}\frac{A'(x)-i}{\left(1 + (A'(x))^2\right)^{1/2}\left[ x-y + i(A(x)-A(y))\right]}
\end{equation}
where $w=x+iA(x),\ z= y+i A(y),$ and  $x, y \in J$ with $x\neq y$.

 \vskip0.1in
  The case when $J$ is the entire real line  has special  relevance in complex analysis because in this case 
   $K_\Gamma $ agrees with the integration
  kernel of the 
  {\bf Cauchy integral operator} associated with  the domain
   \begin{equation}\label{E:domain}
   \Omega := \{ y< A(x),\ x\in J\},
   \end{equation}
namely the operator
 \begin{equation}\label{E:defCauchyop}
f \mapsto \frac{1}{2\pi i}\int_{\bndry \Omega}\! f(w)\, \frac{dw}{w-z},\quad z\notin \text{Supp}(f).
\end{equation}
 As is well known \cite{{Du}, {KS}, {LS}}, the Cauchy integral operator 
 produces and reproduces\footnote{that is, it
 is a {\em projection}:
   $L^p(\bndry\Omega, \sigma)\to H^p(\bndry\Omega, \sigma)$.} functions in the holomorphic Hardy space $H^2(\bndry\Omega, \sigma)$
   (more generally, functions in $H^p(\bndry\Omega, \sigma), 1\leq p\leq \infty$) 
 and 
 it plays a distinguished role in the analysis of the 
 Cauchy-Szeg\Humlaut{o} projection and the Kerzman-Stein operator \eqref{E:KS}.
\vskip0.1in

\noindent To see the 
 connection between \eqref{E:defCauchyop} 
 and the kernel \eqref{E:def-restricted},
  we first
  recall
that 
$dw$ in \eqref{E:defCauchyop}
 is shorthand for the pull-back
$j^{\ast} dw$ where $j: \Gamma \hookrightarrow \mathbb C$ is the inclusion map.
With this in place, and again writing $x+iA(x)$ for $w \in\Gamma$, we have
\begin{align*} 
\frac{1}{2\pi i}\, j^{\ast}dw  &= \frac{1}{2\pi i}\,d(x + iA(x)) \\ &= \frac{1}{2\pi i}\,(1 + i A'(x))\, dx \\ &=
\frac{1}{2\pi}\, \frac{A'(x)-i}{\sqrt{1 + (A'(x))^2}}\, d\sigma(w)\, ,
  \end{align*} 
where $\sigma$ is the arc-length measure for
  $ \bndry\Omega \equiv \Gamma$ whose density is
 $d\sigma(w) = s(x) dx$, with $s(x) = \sqrt{1 + (A'(x))^2}$.
It is now clear that the kernel in \eqref{E:defCauchyop} interpreted as  an integral with respect to the arc-length measure for $\Gamma$, 
agrees with $K_\Gamma $
and we will henceforth  ignore the constant factor $1/2\pi $. 
 
\vskip0.1in
\noindent Note that the restriction to $\Gamma$ of the
 universal 
Cauchy kernel,
 namely  the function
$j^*K_0$, 
corresponds to the case when $A$
 is constant, that is  the situation
  when $\Gamma$ is a horizontal line. On the other hand, for general $\Gamma$ we have $j^*K_0\neq K_\Gamma$. The
 distinction between these two kernels is especially significant in the context of holomorphic Hardy space theory; for instance, the analysis  of the 
 Cauchy-Szeg\Humlaut{o} projection
 performed in e.g., \cite{KS}
  and 
  \cite[
  Theorem 2.1]{LS} relies upon a cancellation of singularities that is enjoyed by the
Kerzman-Stein kernel \eqref{E:KS}
but is {\em not} enjoyed by $j^*K_0(w, z) - \overline{j^*K_0(z, w)}$ unless $\Gamma$ is a horizontal line.
Already in the example of the parabola $\Gamma:= \{x+i x^2\, , x\in\mathbb R\}$, it is easy to see that $j^*K_0(w, z) - \overline{j^*K_0(z, w)}$ has same principal singularity
as $j^*K_0 (w, z)$,
whereas the Kerzman-Stein kernel \eqref{E:KS}
is in fact a  {\em smooth} function of $(w, z) \in \Gamma \times \Gamma$, even along the diagonal $\{w=z\}$. 
\subsection{Main results} 
We establish results of two kinds: local on $\Gamma$, valid for three-tuples of distinct points on $\Gamma$ that are near a point $z_0\in \Gamma$ of non-vanishing curvature; and global on $\Gamma$ (for any three-tuple of distinct points in $\Gamma$). 
 \subsubsection{Sharp local estimates on $\Gamma$} Here we require $\Gamma$ to be of class $C^3$; we show that each of $\mathtt S\,[ \mathrm{Re} K_\Gamma](\mathbf z)$ and $\mathtt S\,[ \mathrm{Im} K_\Gamma](\mathbf z)$ is locally relatively bounded near any point in $\Gamma$ with non-zero signed curvature, but only the former will be non-negative, in fact strictly positive and forcing the latter to be strictly negative. 
 \vskip0.1in
\begin{thm}\label{T:Re and Im local} 
Suppose that $A$ is of class $C^3$ (i.e., $A$ is thrice continuously differentiable), 
and that
 $x_0 \in J$ is such that $A''(x_0) \ne 0$. 
Then for any $\epsilon > 0$, there exists $\delta = \delta(x_0, A, \epsilon)>0$ such that for 
\[ I = (x_0 - \delta, x_0 + \delta)\subset J \qquad \text{ and }  \qquad \Gamma (I)= \{ z = x + i A(x) : x \in I\} \]
the following statements
 hold for any three-tuple $\mathbf z$ of distinct points on
 $$
  \Gamma (I)^3 = \Gamma (I) \times \Gamma (I) \times \Gamma (I). $$
\begin{enumerate}[(a)] 
\item If $\displaystyle{\kappa_0 = A''(x_0)/s(x_0)^3}$
denotes the curvature of \ $\Gamma $ at $z_0= x_0+iA(x_0)$, then
\begin{equation}\label{E:meng-bdd} 
c^2(\mathbf z) = \kappa_0^2 + r(\mathbf z), \quad \text{with}\ 
|\,r(\mathbf z)\,|<\epsilon\, . \ \ 
\end{equation}  
\vskip0.1in
\item We have
\begin{equation}\label{E:Re and Im local}
\Bigl|\,\mathtt S[\mathrm{Re}K_\Gamma ](\mathbf z) - \frac{3}{2}c^2(\mathbf z) \Bigr| < \epsilon, \quad  \Bigl|\,\mathtt S[\mathrm{Im}K_\Gamma ](\mathbf z) +\frac{1}{2}c^2(\mathbf z) \Bigr| < \epsilon. 
\end{equation} 
\end{enumerate} 
\end{thm} 
\vskip0.1in
\noindent Combining the two conclusions of Theorem \ref{T:Re and Im local}, we arrive at the following corollary. 
\begin{cor}\label{C:bddness} With same notations and hypotheses as in Theorem \ref{T:Re and Im local} if, 
furthermore, $\tilde\epsilon>0$ is sufficiently small
 then 
\begin{equation}\label{E:local-relative-bdd}
\Bigl|\,\frac{\mathtt S[\mathrm{Re}K_\Gamma ](\mathbf z)}{c^2(\mathbf z)} -\frac32 \Big|< \frac{\tilde\epsilon}{\kappa_0^2 - \tilde\epsilon}\,;\quad
\Bigl|\,\frac{\mathtt S[\mathrm{Im}K_\Gamma ](\mathbf z)}{c^2(\mathbf z)} +\frac12 \Big|< \frac{\tilde\epsilon}{\kappa_0^2 - \tilde\epsilon}
\end{equation}
for any three-tuple $\mathbf z$ of non-collinear points in $\Gamma(\tilde I)^3$ where $\tilde I = (x_0-\tilde\delta, x_0+\tilde\delta)$ is obtained by applying Theorem \ref{T:Re and Im local} to $\tilde\epsilon$.
\end{cor}
{\bf{Remarks:}}
 \begin{enumerate}[(i)] \item 
Theorem \ref{T:Re and Im local} gives that 
$\mathtt S[\mathrm{Re}K_\Gamma ](\mathbf z)$ satisfies the positivity condition 
{\tt(ii)}
 when $\mathbf z$ is taken in $ \Gamma (I)^3$, but more is true:
the proof will show that $\mathtt S[\mathrm{Re}K_\Gamma ]$
 manifests a phenomenon of ``local superpositivity'' in the sense that for any
  $\mathbf z\in \Gamma (I)^3$,  $\mathtt S[\mathrm{Re}K_\Gamma ](\mathbf z)$ is given by the sum of three positive terms, each
 comparable to $\frac{1}{2} \kappa_0^{2}$. 
 On the other hand
$\mathtt S[\mathrm{Im}K_\Gamma ]$ is strictly negative on $\Gamma (I)^3$,
  in stark contrast with the situation for $\mathtt S[\mathrm{Im}(K_0 )]$ (that is, when $\Gamma$ is a horizontal line).
\vskip0.1in
\item This leads to the following remarkable fact. Recall  that for {\em{any}} kernel function $K(z, w)$
and for {\em any} 
three-tuple of distinct points $\mathbf z$, the quantity $\mathtt S[K](\mathbf z)$ admits the basic split
\begin{equation}\label{E:Symm-real-im}
\mathtt S[K](\mathbf z)\,=\,  \mathtt S[\mathrm{Re} K] (\mathbf z)+ \mathtt S[\mathrm{Im}K](\mathbf z),
\end{equation}
see e.g., \cite[(2.5)]{LPg}.
\vskip0.1in
\noindent However, whereas the split of $\mathtt S[K_0](\mathbf z)$ is perfectly balanced between the real and imaginary parts of $K_0$ i.e., 
\[ \mathtt S[\mathrm{Re} K_0](\mathbf z)\, =\, \frac{1}{2}\mathtt S[K_0](\mathbf z)\, =\, \mathtt S[\mathrm{Im} K_0](\mathbf z)\] 
see \eqref{E:sym-id-0pr} -- \eqref{E:sym-id-0pi},
the split for $\mathtt S[\,K_\Gamma ](\mathbf z)$ with 
$\mathbf z\in \Gamma(I)^3$ is roughly speaking $3/2$ and $-1/2$, respectively, i.e.
\[ \mathtt S[\mathrm{Re}K_\Gamma ](\mathbf z) \approx \frac{3}{2} \mathtt S[\,K_\Gamma ](\mathbf z)
\ \ \text{ and } \ \ \mathtt S[\mathrm{Im}K_\Gamma ](\mathbf z) \approx -\frac{1}{2} \mathtt S[\,K_\Gamma ](\mathbf z)\, .\]
\vskip0.1in
\item Results analogous to Theorem \ref{T:Re and Im local} continue to hold if $A''(x_0)=0$ but some higher order derivative of $A$
is non-vanishing at $x_0$. The proof modifies with very little changes and we have chosen to omit it here.
\vskip0.1in
\item Corollary \ref{C:bddness}  says that both
$\mathtt S[\mathrm{Re}K_\Gamma ]$ and $\mathtt S[\mathrm{Im}K_\Gamma ]$ satisfy the relative boundedness condition {\tt(i)}
in $\Gamma (\tilde I)^3$ (and in fact are themselves locally
 bounded, but see Theorem \ref{T:global-bound-of Re-for-Lip-A'} below for a stronger statement). 
\vskip0.1in
\item In general, the inclusion $\tilde I\subset J$ in Corollary \ref{C:bddness} is strict, and
in the absence of the localization: $\mathbf z\in \Gamma(\tilde I)^3$ there are no definitive results pertaining to condition 
\eqref{E:local-relative-bdd}.
 In Section \ref{SS:failed rel bdd} we give an example of a relatively compact, smooth curve $\Gamma$; a point $z_0\in\Gamma$ with non-zero signed curvature, and  three-tuples 
 $\{\mathbf z_\lambda = (z^1_\lambda; z^2_\lambda; z^3_\lambda)\}_\lambda$ of distinct points on $\Gamma$ such that 
 $z^2_\lambda\to z_0$ but
  $\mathtt S[\mathrm{Re}K_\Gamma] (\mathbf z_\lambda)/c^2(\mathbf z_\lambda)$ and
 $\mathtt S[\mathrm{Im}K_\Gamma] (\mathbf z_\lambda)/c^2(\mathbf z_\lambda)$
 are unbounded.
 \end{enumerate} 
\vskip0.1in
\subsubsection{Qualitative global estimates on $\Gamma$} We further consider two settings:
 \vskip0.05in
  $\bullet$ {\em Curves of class $C^{1,1}$.} This means that $A$ is once differentiable with Lipschitz continuous derivative. We show that
   each of
  $c^2(\mathbf z)$; $\big|\mathtt S\,[ \mathrm{Re} K_\Gamma](\mathbf z)\big|$ and $\big|\mathtt S\,[ \mathrm{Im} K_\Gamma](\mathbf z)\big|$ admits a global upper bound valid for any three-tuple $\mathbf z$ of distinct points on $\Gamma$. Since $c(\mathbf z)$ can vanish, this result does not imply relative boundedness. See Theorem \ref{T:global-bound-of Re-for-Lip-A'} and Corollary \ref{C:global-bound-of Im-for-Lip-A'} below.
  \vskip0.04in
  $\bullet$ {\em Curves of class $C^2$ with fixed concavity}. We prove global non-negativity of 
  $\mathtt S\,[ \mathrm{Re} K_\Gamma](\mathbf z)$ and provide examples to show that there are no definitive results pertaining to the global signature of $\mathtt S\,[ \mathrm{Im} K_\Gamma](\mathbf z)$.
  See Theorem \ref{T:positivity-on-curves-with-fixed-curvature} and the examples in section \ref{SS:Im-pos} below. 

\label{SS: restricted-global-for-special curves}
\begin{thm}\label{T:global-bound-of Re-for-Lip-A'}
Suppose that $A$
is of class $C^{1,1}$
i.e., there exists a constant $M>0$ such that
\begin{equation}\label{A'-Lipschitz}
|A'(x)-A'(y)|\leq M |x-y|\, \quad \mathrm{all}\ x, y\in J.
\end{equation}
Then 
we have that
$$
\big|\, \mathtt S[\mathrm{Re}K_\Gamma ](\mathbf z) \,\big| \leq \frac32 M^2
$$
for any three-tuple $\mathbf z$ of distinct points on $\Gamma$.
\end{thm}
\vskip0.1in
\noindent The order of magnitude of the Lipschitz constant for $A'$
(that is, the quantity $M$ in \eqref{A'-Lipschitz})
 is optimal, as indicated by the following
\begin{lem}\label{L: example-optimal}
Let $A(x)= x^3$.
Then
there are \ \ $0<\delta_0 = \delta_0(A)$ and\ \  $0<c_0<1$, $c_0=c_0(A)$ such that 
$$
\mathtt S[\mathrm{Re}K_\Gamma ](\mathbf z)\geq c_0\,M_\epsilon^2
$$
for $\mathbf z= (-\epsilon\alpha - i \epsilon^3\alpha^3,\, 0,\, \epsilon\beta + i \epsilon^3\beta^3)$,
 for any $\alpha, \beta\in [1/2, 1]$ and any $0<\epsilon<\min\{1, \delta_0\}$. Here $M_\epsilon$ is the Lipschitz constant for the restriction of  $A'(x)$ to the interval
$(-\epsilon, \epsilon)$.
\end{lem}
\noindent Note that Lemma \ref{L: example-optimal} also shows that Theorem \ref{T:Re and Im local}  gives a sufficient, but not necessary condition for the local positivity of $\mathtt S[\mathrm{Re}K_\Gamma ]$.
\begin{lem}\label{L:global-bound-of Menger-for-Lip-A'}
With same hypotheses as Theorem \ref{T:global-bound-of Re-for-Lip-A'}, we have that
$$
c^2(\mathbf z)\leq 8 M^2
$$
for any three-tuple $\mathbf z$ of distinct points on $\Gamma$.
\end{lem}
\vskip0.1in
\begin{cor}\label{C:global-bound-of Im-for-Lip-A'}
With same hypotheses as Theorem \ref{T:global-bound-of Re-for-Lip-A'}, we have that
$$
\big|\,\mathtt S[\mathrm{Im}K_\Gamma ](\mathbf z)\big|\leq \left(\!8 + \frac32\right)\!M^2 \quad
$$
for any three-tuple $\mathbf z$ of distinct points on $\Gamma$.
\end{cor}
\begin{thm}\label{T:positivity-on-curves-with-fixed-curvature}
Suppose that $A$
is of class $C^2$, and that $A''$ does not change sign on $J$ e.g.,
$A''(x)\geq 0$ for all $x\in J$ (alt. $A''(x)\leq 0$ for all $x\in J$).
Then 
\begin{equation}
\mathtt S[\mathrm{Re}K_\Gamma ](\mathbf z) \geq 0
\end{equation}
 for any three-tuple $\mathbf z$ of distinct points on $\Gamma$.
\end{thm}
{\bf{Remarks: }} 
\begin{enumerate}[(i)]
\item In view of conclusion \eqref{E:Re and Im local} in Theorem \ref{T:Re and Im local}, it makes sense to ask whether the inequality: $\mathtt S[\mathrm{Im}K_\Gamma ](\mathbf z)\leq 0$ can hold for any three-tuple $\mathbf z$ of distinct points on a curve satisfying the hypotheses of Theorem \ref{T:positivity-on-curves-with-fixed-curvature}: in section \ref{SS:Im-pos} below we answer this question in the negative by showing that the parabola
 $\Gamma =\{x+i\, x^2/2, x\in\mathbb R\}$ admits three-tuples $\mathbf z$ of distinct points
   such that $\mathtt S[\mathrm{Im}K_\Gamma ](\mathbf z)> 0$ .
\vskip0.1in
\item The assumption of fixed concavity in Theorem \ref{T:positivity-on-curves-with-fixed-curvature} is necessary: in section \ref{SS: fail-global-pos} we show that the cubic $\Gamma =\{x+i x^3, \ x\in\mathbb R\}$ admits  three-tuples $\mathbf z$ of distinct points
 for which $\mathtt S[\mathrm{Re}K_\Gamma ](\mathbf z)< 0$.
\vskip0.1in
\item 
While our global results hold for three-tuples 
 that lie on the curve
 $\Gamma$, the original Melnikov miracles \eqref{E:sym-id-0pr} and \eqref{E:sym-id-0pi}
  are in fact universally global in that they
 are valid for three-tuples
  that may lie {\em anywhere} in $\mathbb C$. 
It is thus meaningful
 to ask whether analogues of Theorem \ref{T:global-bound-of Re-for-Lip-A'}, Corollary \ref{C:global-bound-of Im-for-Lip-A'}, or Theorem \ref{T:positivity-on-curves-with-fixed-curvature}
can be stated that hold for
arbitrary three-tuples 
 in $\mathbb C$. To this end we consider the following family of kernels $\{K_h\}_h$ 
\begin{equation}\label{E:Kh-def}
K_h(w, z) := \frac{e^{i h(w)}}{w-z}
\end{equation}
 parametrized by globally defined, continuous functions $h: \mathbb C \rightarrow \mathbb R$. The family $\{K_h\}_h$ is especially interesting to us because for any $C^1$-smooth curve $\Gamma$ 
  the kernel $K_\Gamma$ is realized as $j^*K_h$ for at least one such $h$. To see this, 
  note that the domain $\Omega$ determined by  $\Gamma$ as in \eqref{E:curve} and \eqref{E:domain},
    is contained in the complement of a ray, thereby granting the existence of a continuous branch of the logarithm of the function
    $$
  x\mapsto  
  \frac{A'(x)-i}{(1+(A'(x))^2)^{1/2}},\quad x\in J.
    $$ 
    Applying such
     logarithm
     gives
      a continuous function: $ \Gamma \to \mathbb R$ which we call $\varphi$. Any extension of $\varphi$  to a continuous\footnote{for instance,  extending $A$ to a $C^1$-function: $\tilde{A}: \mathbb R\to\mathbb R$ and then letting $h(x+iy):= \varphi(x+i\tilde{A}(x))$
gives a continuous function $h:\mathbb C\to\mathbb R$ that is constant along each vertical line.} 
 $h: \mathbb C\to \mathbb R$
   produces a kernel  $K_h$ in the family \eqref{E:Kh-def} whose restriction to $\Gamma$ agrees with $K_\Gamma$.

It turns out
 that $\mathtt S[\mathrm{Re}K_h ]$
  and
  $\mathtt S[\mathrm{Im}K_h]$ satisfy no universally global phenomena
 for any continuous $h:\mathbb C\to \mathbb R$, not even for $h$ chosen so that  $j^*K_h = K_\Gamma$ where
$\Gamma$ satisfies the stronger hypotheses of class $C^{1,1}$ (Theorem \ref{T:global-bound-of Re-for-Lip-A'}) or class $C^2$ (Theorem \ref{T:positivity-on-curves-with-fixed-curvature}): this is proved with
 techniques that
 are similar in spirit to recent work of Chousionis-Pratt \cite{CP}
 and Chunaev \cite{C}.
 The precise statements are given in section \ref{S:appendix},
 see Theorem \ref{T:L-infty} - Theorem \ref{T:L-infty-H}.

\end{enumerate} 
\vskip0.1in

\noindent {\bf An open problem.} {\em  Does the stronger assumption: $$A''(x)\geq c>0\quad \text{for all }\ x\in J\quad  \text{(alt.}\ A''(x)\leq c<0\ \text{for all}\ x\in J)$$ give that
\begin{equation}\label{E:Re-pos-str}
\mathtt S[\mathrm{Re}K_\Gamma ](\mathbf z) \geq \alpha\, c^2(\mathbf z)
\end{equation}
for some $\alpha = \alpha(\Gamma)>0$ and for all three-tuples $\mathbf z$ of distinct points on $\Gamma$?}
\vskip0.1in
\noindent This statement seems much harder to prove than Theorem \ref{T:positivity-on-curves-with-fixed-curvature} (whose proof is remarkably simple). Note that an answer in the positive would shed some light on the signature of $\mathtt S[\mathrm{Im}K_\Gamma ](\mathbf z)$ because it would imply that
$$
\mathtt S[\mathrm{Im}K_\Gamma ](\mathbf z) \leq (1-\alpha) c^2(\mathbf z)
$$
for all three-tuples $\mathbf z$ of distinct points on $\Gamma$.  In the example of the parabola:  $\Gamma =\{x+i x^2/2, x\in\mathbb R\}$, an elementary but non-trivial calculation\footnote{we are grateful to M. Putinar, E. Wegert and A. Weideman for assisting with these computations.} gives that \eqref{E:Re-pos-str} is true with $\alpha = 1/2$; the general case remains unanswered.
\subsection{Conclusion} The extensive literature on this subject 
indicates that most kernels do not satisfy basic estimates such as  {\tt(i)} and {\tt (ii)} globally in $\mathbb C$:  the families $\{\mathrm{Re}\Kh\}_h$ and $\{\mathrm{Im}\Kh\}_h$ are no exception.
 On the other hand,  given any rectifiable curve $\Gamma$ with regularity prescribed  as in  our main results, it turns out that certain members in
  $\{\mathrm{Re}\Kh\}_h$ (those for which $j^*\Kh = K_\Gamma$) do satisfy $\Gamma$-restricted versions of  {\tt(i)} and {\tt (ii)}, whereas their counterparts in $\{\mathrm{Im}\Kh\}_h$ will satisfy $\Gamma$-restricted versions of {\tt(i)} though not necessarily of {\tt(ii)}. 
  
 As mentioned earlier, symmetrization techniques have so far been used primarily to study $L^2$-regularity of Calder\`on-Zygmund operators. One would like to know whether effective curvature methods can be developed to prove certain finer properties  of the Kerzman-Stein operator (such as compactness in $L^p(\Gamma, \sigma)$ \cite{Semmes1983}) 
  and of various Kerzman-Stein-like operators that are known to bear
  upon 
   the $L^p$-regularity problem for
  the Cauchy-Szeg\Humlaut{o} projection, the Bergman projection and other holomorphic singular integral operators: the $\Gamma$-restricted estimates obtained here are a first step in this direction; we plan to pursue 
  the subsequent steps elsewhere.
  \vskip0.1in
  
  \noindent {\bf One last remark.} {\em What if in place of the symmetrized form \eqref{E:K-sym}  one had considered the following variant:
  \begin{equation} 
\tilde{\mathtt S}\,[K](\mathbf z) :=   \sum\limits_{\sigma\in S_3}
  K(z_{\sigma(2)},\, z_{\sigma(1)})\, 
  \overline{K(z_{\sigma(3)},\, z_{\sigma(1)})} ,
\end{equation} 
whose choice is also legitimate  because $\tilde{\mathtt S}\,[K_0](\mathbf z) = {\mathtt S}\,[K_0](\mathbf z)$?
}
\vskip0.1in
  Setting $ K^*(w, z):= \overline{K(z, w)}$ it is easy to see that 
$ \tilde{\mathtt S}\,[K](\mathbf z) \ =\ \mathtt S\,[K^*](\mathbf z)$.
In section \ref{S:last} we prove failure of the basic estimate {\tt (i)} for the family $\{K_h^*\}_h$ globally in $\mathbb C$, see Proposition \ref{P:symm-id-K-star}  and Theorem \ref{T:L-infty-H}. The analysis in the $\Gamma$-restricted setting will be the object of forthcoming work.
\subsection{Organization of this paper} In section \ref{S:Prelimns} we collect a few auxiliary facts needed to prove the main results, whose proofs are given in section \ref{S:proofs-main}. All examples pertaining to the sharpness of the main results are 
detailed
in section \ref{S:Examples}. Finally, section \ref{S:appendix} is an appendix where we 
 collect all the relevant statements for the family of kernels \eqref{E:Kh-def}.

 \subsection{Acknowledgements.} The authors were supported 
by awards no.\ DMS-1503612  and DMS-1901978 from the National Science Foundation USA; and
a Discovery grant from the National Science and Engineering Research
Council of Canada. Part of this work took place {\em (a)} at the Mathematical Sciences Research Institute in Berkeley, California, where the authors were in residence during a thematic program in the spring of 2017; {\em (b)} at the Park City Mathematics Institute in July 2018, during the thematic program in harmonic analysis, and {\em (c)} at the Isaac Newton Institute for Mathematical Sciences, where the first-named author was in residence in Fall 2019 during the program {\em Complex Analysis: Theory and Applications} (EPSRC grant no. EP/R014604/1). We thank the institutes and the programs organizers for their generous support and hospitality.  Last but not least, we are very grateful to the reviewer for helpful feedback and for pointing out relevant references.

  \section{Preliminaries}\label{S:Prelimns}
\noindent We begin by recording representations
for $\mathtt S\,[\text{Re} K_\Gamma ](\mathbf z)$, $\mathtt S[\,\mathrm{Im}K_\Gamma ](\mathbf z)$ and $c(\mathbf z)$ that hold when $\mathbf z$ is a three-tuple of distinct points on $\Gamma$, that is for  $\mathbf z =(z_1, z_2, z_3)$ with $z_j=x_j+iA(x_j)\in \Gamma$, $j=1, 2, 3$ and distinct $x_1, x_2, x_3$. 
\begin{lem} \label{L:sym-id-curves} $A$
be of class $C^1$. Then
the symmetrized forms of $\mathrm{Re}K_\Gamma  (\mathbf z )$ and $\mathrm{Im}K_\Gamma  (\mathbf z)$ admit the following representations at  any three-tuple $\mathbf z$ of distinct points on $\Gamma$:
\begin{align}\label{E:RE-graph}
&
\begin{aligned} 
\mathtt S[\,\mathrm{Re}K_\Gamma \,](\mathbf z) &= 2 \sum\limits_{\stackrel{j=1}{k<l}}^3\, \frac{1}{s^2(x_j)\ell^2_k\,\ell^2_l}\,
\Big[A'(x_j)(x_j-x_k)-\big(A(x_j)-A(x_k)\big)\!\Bigr] \times \\ 
& \hskip1.4in \Bigl[A'(x_j)(x_j-x_l)-\big(A(x_j)-A(x_l)\big)\!\Bigr]\, ;
\end{aligned} \\
& \label{E:IM-graph}
\begin{aligned} 
\mathtt S[\,\mathrm{Im}K_\Gamma \,](\mathbf z)\! &= 2  \sum\limits_{\stackrel{j=1}{k<l}}^3
\frac{1}{s^2(x_j)\ell^2_k\,\ell^2_l}
\Bigl[(x_k-x_j)  + A'(x_j)\big(A(x_k)-A(x_j)\big)\!\Bigr] \times \\ 
&\hskip1.4in \Bigl[(x_l-x_j)  + A'(x_j)\big(A(x_l)-A(x_j)\big)\!\Bigr].
\end{aligned} 
\end{align} 
\end{lem} 
\noindent As in \cite{LPg}, here we have set $\{j, l, k\} = \{1, 2, 3\}$ and $l, k\in \{1, 2, 3\}\setminus \{j\}$, and we have adopted the shorthand
$$
s^2(x_j) = 1 + (A'(x_j))^2\,; \quad 
\ell^2_j = (x_l-x_k)^2 + (A(x_l)-A(x_k))^2\,.
$$
\begin{proof}
 First we recall  that if $H(w, z)$ is {\em real-valued}, then 
 \begin{equation}\label{E:Symm-simpl3}
  \mathtt S\,[H](\mathbf z)=
  2\sum\limits_{\stackrel{j=1}{k<l}}^3
  H(z_j,\, z_k)\, 
  H(z_j,\, z_\ell)\, ,
  \end{equation} 
see e.g., \cite{LPg}. Next we note that \eqref{E:def-restricted} gives
\begin{align*} 
\mbox{Re} K_\Gamma (w, z)\, &=\, \frac{A'(x)(x-y)- \big(A(x)-A(y)\big)}{s(x)|w-z|^2}\, \text{ and } \\ 
\mbox{Im} K_\Gamma (w, z)\, &=\, \frac{y-x + A'(x)\big(A(y)-A(x)\big)}{s(x)|w-z|^2}
\end{align*} 
for distinct points $w=x+iA(x), z =y+iA(y)$ in $\Gamma$.
The conclusion now follows by plugging these expressions in
\eqref{E:Symm-simpl3}.
\end{proof}
\vskip0.1in
\begin{lem} \label{L:menger on curves} Let
 $A$
be continuous, and let $\mathbf z = \big(u +iA(u); x+iA(x); v+iA(v)\big)$ be any three-tuple  of distinct points on $\Gamma$.
Then the Menger curvature of $\mathbf z$ admits the following representation:
\begin{equation}\label{E:menger on curves}
c^2(\mathbf z) = 
\frac{4\,\left[A(u)(x-v) + A(x)(v-u) + A(v)(u-x)\right]^2}
{\ell^2_u\, \ell^2_x\, \ell^2_v}.
\end{equation}
\end{lem}

\noindent As before, here we have adopted the shorthand: $\ell^2_u = (v-x)^2 + (A(v)-A(x))^2$, etc.

\vskip0.1in

\begin{proof}
If the three distinct points are collinear then the conclusion is immediate because each side of \eqref{E:menger on curves} is easily seen to be equal to zero. Suppose next that the three points are not collinear: 
by the invariance of the numerator of \eqref{E:menger on curves} under the permutations of $\{u, x, v\}$ we may assume without loss of generality that $u<x<v$. Then there are two cases to consider, depending on whether the point $(x, A(x))$ lies below or above the line segment joining  $(u, A(u))$ and $(v, A(v))$. In either case we may assume without loss of generality, that
$$
A(u)>0;\ A(x)>0;\ A(v)>0.
$$
(This is because Menger curvature is invariant under translations, and the above condition is achieved by a translation along the vertical axis.) The desired conclusion then follows by employing the well-known formula
\cite[(3.1)]{T}
   \begin{equation}\label{E:Marea-1}
 c(\mathbf z)\,=\, \frac{4\,\text{Area}(\Delta (\mathbf z))}{\ell_a\ell_b\ell_c},
 \end{equation}
and by expressing the area of the triangle  $\Delta (\mathbf z)$ as an appropriate linear combination of
areas of parallelograms whose vertices belong to the set 
$$\{ (u, 0); (x, 0); (v, 0); (u, A(u)); (x, A(x)); (v, A(v))\}.$$
\end{proof}
\noindent Next we provide an elementary lemma that rules out the possibility of collinearity for three-tuples $\mathbf z$ of distinct points on $\Gamma$ in the vicinity of points with non-zero signed curvature.

\begin{lem}\label{L:Rolle} 
Let $A$ be
 of class $C^2$. 
Then any three-tuple $\mathbf z$ of distinct points on $\Gamma$
  that are in the vicinity of a point $z_0\in \Gamma$ whose curvature $\kappa_0$ is non-zero, are 
  non-collinear.

\end{lem}
\begin{proof}  We need to show that for any $x_0\in J$ such that $A''(x_0)\neq 0$ there is $\delta>0$ with the property that  
for any $u, x, v\in I := (x_0-\delta, x_0+\delta)$
the points
$ u +iA(u); x+iA(x); v + iA(v)$ are not collinear.
Suppose, by contradiction, that there are $x_0\in J$ and $u_n<v_n<w_n\to x_0$ such that the points $P_n:=(u_n, A(u_n)); Q_n:= (v_n, A(v_n))$ and $R_n:= (w_n, A(w_n))$ are collinear. Then the slopes of the line segments joining any two such points must be equal, giving us
$$
\frac{A(v_n)-A(u_n)}{v_n-u_n}\, =\, \frac{A(w_n)-A(v_n)}{w_n-v_n}\quad \text{for all}\ n.
$$
By the mean value theorem it follows that there are $x_n$ and $y_n$ with $u_n < x_n< v_n< y_n< w_n$ and such that
$$
A'(x_n) = A'(y_n)\quad \text{for all}\ n.
$$
Applying Rolle's theorem to $f(x):= A'(x)$ we conclude that for each $n$ there is  $z_n$ with $x_n< z_n< y_n$ and such that 
$$
A''(z_n)=0\quad \text{for all}\ n,
$$
leading us to a contradiction since $A''(z_n)\to A''(x_0)\neq 0$.
\end{proof}
\noindent {\bf Remark.} The same strategy of proof also gives the following global version of Lemma \ref{L:Rolle}:
{\em If $A$: 
is of class $C^2$ and $A''(x)\neq 0$ for all $x\in J$, then any three-tuple 
of distinct points on $\Gamma$ 
 are non-collinear.}

\vskip0.1in

\noindent In closing this section we detail a few lemmas that help
 to keep track of the effect of the assumed regularity of $A(x)$ in the proofs of our main results. 
\begin{lem}\label{L:1}
Let $A$ be
of class $C^2$.
Then for any $x_0\in J$ and $\epsilon >0$ there is $\delta>0$ such that
\begin{equation}\label{E:L1}
|A(v)-A(u)-A'(x_0)(v-u)|\leq\epsilon|v-u|\quad
\end{equation}
whenever $u, v\in I = (x_0-\delta, x_0+\delta)$.
\end{lem}
\begin{proof}
Taylor's theorem grants the existence of $\delta_1>0$ (which we may take to be finite) such that
\begin{equation}\label{E:Taylor a}
A(v) - A(u) - A'(x_0)(v-u) = R_1(v) - R_1(u)
\end{equation}
for any $u, v\in I_{\delta_1}(x_0)$, where 
\begin{equation}\label{E:Tayl-rem}
R_1(y) =\int\limits_{x_0}^y (y-t) A''(t)\, dt. 
\end{equation}
 We claim that there is $0<\delta\leq \delta_1$ such that
\begin{equation}\label{E:Taylor b}
|R_1(v) - R_1(u)|\leq \epsilon |v-u|\quad \text{for any}\ \ u, v\in I_\delta(x_0).
\end{equation}
The claim is trivial when $u=v$  and we henceforth assume that $u<v$. 
It follows from \eqref{E:Tayl-rem} that
$$
R_1(v) - R_1(u) = (v-u)\int\limits_{x_0}^u A''(t)dt + \int\limits_u^v(v-t)A''(t)dt.
$$
Integrating the second integral by parts, and then applying the mean-value theorem give that
$$
R_1(v) - R_1(u) = (v-u)\int\limits_{x_0}^u A''(t)dt\ - A'(u)(v-u) + A'(\xi)(v-u)
$$
for some $\xi$ with $u<\xi<v$ and for any $u, v\in I_{\delta_1}(x_0)$. Now
  for
$$
\|A''\|_\infty := \|A''\|_{L^\infty (I_{\delta_1}(x_0))}\,
$$
we see that the above gives
$$
|R_1(v) - R_1(u)|\leq |v-u| \|A''\|_\infty 2\delta\leq \epsilon |v-u|
$$
as soon as we choose $0<\delta \leq \min\{\delta_1, \epsilon/(2\|A''\|_\infty)\}$.
\end{proof}
\vskip0.1in
\begin{cor}\label{C:CtoL1} 
Let $A$
be of class $C^3$
 and suppose that $A''(x_0)\neq 0$. Then for any $\epsilon >0$ there is $\delta>0$ such that
\begin{equation}\label{E:CtoL1}
A'(u)-A'(v) = A''(x_0)(u-v)\left(1+\mu (u, v)\right)\ \ \text{with}\ \ |\,\mu(u, v)|<\epsilon\, ,
\end{equation}
whenever $u, v\in I = (x_0-\delta, x_0+\delta)$.
\end{cor}
\begin{proof}
Applying Lemma \ref{L:1} to the function $A'(x)$ (which is of class $C^2$) we obtain $\delta >0$ such that
$$
A'(u)-A'(v) = A''(x_0)(u-v) + \tilde R_1(v) - \tilde R_1(u)
$$
 where
$$
|\tilde R_1(v) - \tilde R_1(u)|\leq \epsilon |v-u|
$$
for any $u, v \in I_{\delta} (x_0)$; see \eqref{E:Taylor a} and \eqref{E:Taylor b}.
Thus the conclusion holds with 
$$
\mu(u, v) = \frac{\tilde R_1(v) - \tilde R_1(u)}{A''(x_0)(v-u)}.
$$
\end{proof}
\begin{lem}\label{L:2} 
Let $A$
be of class $C^3$
and suppose that $A''(x_0)\neq 0$. Then for any $\epsilon >0$ there is $\delta>0$ such that
\begin{equation}\label{E:L2}
\int\limits_u^v\!\!\big(A'(t)-A'(u)\big)dt = A''(x_0)\frac{(u-v)^2}{2}\big(1 + R(u, v)\big)\,
\text{with}\ |R(u, v)|<\epsilon
\end{equation}
whenever $u, v\in I = I_{\delta}(x_0) := (x_0-\delta, x_0+\delta)$.
\end{lem}
\begin{proof}
We may assume without loss of generality that $u\neq v$ and further, that 
$u<v$
Applying Lemma \ref{L:1} to the function $A'(x)$ (which is of class $C^2$) we obtain $\delta >0$ such that for any $t, u \in I_{\delta} (x_0)$ we have
\begin{equation}\label{E:pre-int}
A'(t)-A'(u) = A''(x_0)(t-u) + \tilde R_1(t) - \tilde R_1(u)
\end{equation}
and
\begin{equation}\label{E:rem-tilde}
|\tilde R_1(t) - \tilde R_1(u)|\leq \epsilon |t-u|\,;
\end{equation}
see \eqref{E:Taylor a} and \eqref{E:Taylor b}. Next we take 
$v\in I_{\delta}(x_0)$, and integrate both sides of \eqref{E:pre-int} over the sub-interval $(u, v)\subset I_\delta(x_0)$:
$$
\int\limits_u^v (A'(t)-A'(u))dt = 
\frac{A''(x_0)}{2}(v-u)^2 + 
\int\limits_u^v \big(\tilde R_1(t) - \tilde R_1(u)\big)dt\, .
$$
It follows from \eqref{E:rem-tilde} that
$$
\left|\,\int\limits_u^v \!\!\big(\tilde R_1(t) - \tilde R_1(u)\big)dt\right|\leq 
\epsilon \int\limits_u^v |t-u|\, dt = 
\frac{\epsilon}{2}(v-u)^2
$$
for any $u, v \in I_\delta(x_0)$. 
Thus the conclusion holds with 
$$
R(u, v) = \frac{\displaystyle{\int\limits_u^v\left(\tilde R_1(t) - \tilde R_1(u)\right)dt}}{A''(x_0)(v-u)^2}.
$$
\end{proof}

\begin{lem}\label{L:3} 
Let
$A$ be
of class $C^2$.
 Then for any $x_0\in J$ and any $\epsilon >0$ there is $\delta>0$ such that for any $u, v\in I = (x_0-\delta, x_0+\delta)$ we have
\begin{equation}\label{E:L3}
\frac{(v-u)^2}{[(v-u)^2 + (A(v)-A(u))^2]}\, =\, \frac{1}{s^2(x_0)}\left(1+\rho(u, v)\right)
\end{equation}
where\quad
$$
|\rho (u, v)|<\epsilon (1+2|A'(x_0)|)\, .
$$
\end{lem}
\begin{proof}
A straightforward application of Lemma \ref{L:1} gives $\delta>0$ such that
$$
\frac{(v-u)^2 + (A(v)-A(u))^2}{(v-u)^2} = s^2(x_0)\left(1+ \frac{S(u, v)}{(v-u)^2})\right)
$$
for any $u, v\in I_\delta(x_0)$, where (with the same notations as the proof of Lemma \ref{L:1})
$$
S(u, v) := (R_1(v)-R_1(u))^2 + 2A'(x_0)(v-u)(R_1(v)-R_1(u))
$$
has $|S(u, v)|\leq \epsilon (1+ 2|A'(x_0)|)(v-u)^2$. Taking reciprocals, we obtain the desired conclusion 
by chosing
$$
\rho(u, v) := -\, \frac{S(u, v)}{\displaystyle{1+\frac{S(u, v)}{(v-u)^2}}}.
$$

\end{proof}
\begin{lem}\label{L:4}
 Let $A$ be
 of class $C^2$.
 Then for any $x_0\in J$ and any $\epsilon >0$ there is $\delta>0$ such that
\begin{equation}\label{E:L4}
\frac{1}{s^2(x)} \ =\ \frac{1}{s^2(x_0)}\left(1+\beta (x)\right)
\text{with}\ \ |\beta (x)|<\epsilon\, ,
\end{equation}
whenever $x\in I = I_{\delta}(x_0) :=  (x_0-\delta, x_0+\delta)$.
\end{lem}
\begin{proof}
Taylor's theorem gives $\delta_1>0$ (which we may choose to be finite) such that for any $x\in I_{\delta_1}(x_0)$ we have
$$
f(x) = f(x_0)\left( 1 + \frac{R_0(x)}{f(x_0)}\right)
$$
with $\displaystyle{|R_0(x)|\leq \|f'\|_{L^\infty(I)}|x-x_0|}$, where $f(x) = 1/s^2(x)$. 
Since
$$
\|f'\|_{L^\infty(I)} \leq 2\|A'\|_{L^\infty(I)}\|A''\|_{L^\infty(I)},
$$
 the conclusion holds if we choose $\beta(x) = R_0(x)s^2(x_0)$ and 
 $$\displaystyle{0<\delta\leq \min\left\{\delta_1,\ \frac12\frac{s^2(x_0)}{\|A'\|_{L^\infty(I)}\|A''\|_{L^\infty(I)}}\right\}}.$$
\end{proof}

\section{Proofs of the main results}\label{S:proofs-main}

\subsection{Proof of Theorem \ref{T:Re and Im local}} The proof of conclusion (a) will follow from 
finitely many applications of Lemmas \ref{L:Rolle} 
 through \ref{L:4}. We henceforth set $\delta > 0$ to be the minimum among  the positive numbers $\delta$ obtained in those lemmas. Let $I=(x_0-\delta, x_0+\delta)$ and let
${\mathbf{z}} := (u+iA(u), x+i A(x), v+i A(v))$ be any three-tuple of distinct points on $\Gamma(I)^3$.
By Lemma \ref{L:Rolle} such points
 are non-collinear, thus $c^2(\mathbf z)$ is strictly positive and it admits the representation
 \eqref{E:menger on curves}. We may assume without loss of generality that
 $u<x<v$.  We write $v-u = (v-x) + (x-u)$ and obtain that the numerator in the righthand side of  
 \eqref{E:menger on curves} equals 4 times 
 $$
 \!\left[(v-x)\!\int\limits_u^x\! A'(t)dt\, - \,(x-u)\!\int\limits_x^v\!(A'(t))dt
 \right]^{\!2}\, .
 $$
 Adding and subtracting the quantity $A'(u)$ from the first integral, and the quantity $A'(x)$ from the second integral, leads us to the following expression for the numerator in the righthand side of  
 \eqref{E:menger on curves}:
 \begin{align*}
 \!\Bigl[(v-x)\!\int\limits_u^x\!(A'(t)-A'(u))dt\, - (v-x)(x-u)\big(A'(x)-A'(u)\big)\, - (x-u)\!\int\limits_x^v\!(A'(t)-A'(x))dt
 \Bigr]^{\!2}\, .
 \end{align*}
Applying Lemma \ref{L:2} to the each of the two integral terms, and Corollary \ref{C:CtoL1} to the remaining term, we see that the above quantity equals
\begin{align*} 
& \bigl[(v-x)(x-u)A''(x_0) \bigr]^2\! \times \\ &\hskip1in \left[
(x-u)\big(1+ R(u, x)\big) -2(x-u)\big(1+\mu(u, x)\big) - (v-x)\big(1+ R(x, v)\big)
\right]^2 \!\! \\ 
= & \bigl[(v-x)(x-u)(v-u) A''(x_0) \bigr]^2\, \\ &\hskip1in + 
\big[C^2(u, x, v)-2(v-u)\,C(u, x, v)\big] [(v-x)(x-u)A''(x_0)]^2
\end{align*} 
where
\begin{equation}\label{E:DefC}
C(u, x, v)\, :=\, (x-u)\big[R(u, x) -2\mu (u, x)\big] - (v-x)R(x, v)\, ,
\end{equation}
and $|\mu (u, x)|<\epsilon$ and $|R(u, x)|<\epsilon$.
Furthermore, each of $(x-u)$ and $(v-x)$ is less than $(v-u)$, thus 
\begin{equation}\label{E:Cbound}
|C^2(u, x, v)-2(v-u)\,C(u, x, v)|\leq 8\epsilon(1+2\epsilon)(v-u)^2.
\end{equation}
Plugging the above in \eqref{E:menger on curves} we obtain
$$
c^2(\mathbf z) = E_1(u, x, v)\ +\ 
(A''(x_0))^2 \times
$$
$$
\times
\left[\frac{(v-x)^2}{(v-x)^2+ (A(v)-A(x))^2}\right]\!\!
\left[\frac{(x-u)^2}{(x-u)^2+ (A(x)-A(u))^2}\right]\!\!
\left[\frac{(v-u)^2}{(v-u)^2+ (A(v)-A(u))^2}\right]\!
$$
\vskip0.07in
with
$$
E_1(u, x, v):= (A''(x_0))^2 \times
$$
$$
\times
\left[\frac{(v-x)^2}{(v-x)^2+ (A(v)-A(x))^2}\right]\!\!
\left[\frac{(x-u)^2}{(x-u)^2+ (A(x)-A(u))^2}\right]\!\!
\left[\frac{C^2(u, x, v)-2(v-u)\,C(u, x, v)}{(v-u)^2+ (A(v)-A(u))^2}\right].
$$
\vskip0.1in
\noindent We now apply Lemma \ref{L:3} to each of the fractional factors in  $E_1(u, x, v)$ 
 (use \eqref{E:Cbound} to deal with the third factor) and obtain that
 $$
 |E_1(u, x, v)|\leq \left(\frac{A''(x_0)}{s^3(x_0)}\right)^{\!2}\!\!\big|(1+\rho(x, v))
 (1+\rho(u, x))
 (1+\rho(u, v))\big| 8\epsilon (1+2\epsilon)\leq \epsilon'.
 $$
 One more application of Lemma \ref{L:3} gives
\begin{align*} 
c^2(\mathbf z) &=
\left(\frac{A''(x_0)}{s^3(x_0)}\right)^{\!2}\!\!
\big(1+\rho(x, v)\big)\big(1+\rho(u, x)\big)\big(1+\rho(u, v)\big)\, +\, E_1(u, x, v)\  \\ 
&= \kappa_0^2 + E_2(u, x, v) \, ,\quad \mathrm{with}\ |E_2(u, x, v)|<\epsilon\,.
\end{align*} 
The proof of part (a) in Theorem \ref{T:Re and Im local} is concluded.
\vskip0.1in

\noindent To prove conclusion (b), we begin by making the following claim:
\begin{equation}\label{E:kappa}
\mathtt S[\,\mathrm{Re}K_\Gamma \,](\mathbf z) = \frac32 \kappa_0^2\, +\, \lambda(\mathbf z)\quad 
\mathrm{with}\ |\lambda (\mathbf z)|<\epsilon
\end{equation}
whenever $\mathbf z\in \Gamma (I)^3$. To see this, recall that Lemma \ref{L:sym-id-curves} gives
\begin{align*} 
\mathtt S[\,\mathrm{Re}K_\Gamma \,](\mathbf z) &= 2 \sum\limits_{\stackrel{j}{k<l}}\, \frac{1}{s^2(x_j)\ell^2_k\,\ell^2_l}\,
\Big[A'(x_j)(x_j-x_k)-\big(A(x_j)-A(x_k)\big)\!\Bigr] \\ &\hskip1.5in  \times 
 \Bigl[A'(x_j)(x_j-x_l)-\big(A(x_j)-A(x_l)\big)\!\Bigr]\ \\
&=2\sum\limits_{\stackrel{j}{k<l}}\, \frac{1}{s^2(x_j)\ell^2_k\,\ell^2_l}
\left[\ \int\limits_{x_k}^{x_j}(A'(t)-A'(x_j))dt\right]
 \times 
 \left[\ \int\limits_{x_\ell}^{x_j}(A'(t)-A'(x_j))dt\right]\, .
\end{align*} 
By Lemmas \ref{L:2}, \ref{L:3} and \ref{L:4}  the latter equals
\begin{align*} 
\frac{1}{2}(A''(x_0))^2 &\sum\limits_{\stackrel{j}{k<l}}\, \frac{1}{s^2(x_j)}
\frac{(x_j-x_k)^2}{\ell^2_\ell}\frac{(x_j-x_\ell)^2}{\ell^2_k} (1+R_{jk})(1+ R_{j\ell}) \\ 
&=
\frac{1}{2}\left(\frac{A''(x_0)}{s^2(x_0)}\right)^{\!2}\sum\limits_{\stackrel{j}{k<l}}\, 
\frac{1}{s^2(x_j)}(1+\rho_{jk})(1+\rho_{j\ell})(1+R_{jk})(1+ R_{j\ell}) \\ 
&=\frac{1}{2}\left(\frac{A''(x_0)}{s^3(x_0)}\right)^{\!2}\sum\limits_{\stackrel{j}{k<l}}\, 
(1+\beta_j)(1+\rho_{jk})(1+\rho_{j\ell})(1+R_{jk})(1+ R_{j\ell}) \\ 
&=\frac12 \kappa_0^2\left[3 + \eta (\mathbf z)\right] \quad \mathrm{with}\ |\eta (\mathbf z)|<\epsilon
\end{align*} 
whenever $\mathbf z\in \Gamma(I)^3$.
This ends the proof of \eqref{E:kappa}. The conclusion of the proof of part (b) is now an immediate consequence of \eqref{E:kappa} along with part (a) and the familiar 
split \eqref{E:Symm-real-im}.
\vskip0.1in
\subsection{Proof of Theorem \ref{T:global-bound-of Re-for-Lip-A'}}
For notational simplicity, we write $z_j \in \Gamma^3$ as $z_j = x_j + i A(x_j) =  x_j + i A_j$ and similarly set  $A'_j = A'(x_j), \; s_j = s(x_j)$. 
Recall from Lemma \ref{L:sym-id-curves} that
\begin{align*}
\frac{1}{2} \mathtt S[\text{Re}(K_\Gamma )](\mathbf z) &= \sum\limits_{\stackrel{j=1}{k<l}}^3
\frac{1}{\ell^2_k\ell^2 s^2_j}
 \left[ A_j'(x_j - x_k) - (A_j - A_k)\right] \times  \left[A_j'(x_j - x_l) - (A_j - A_\ell) \right] \\ 
&= \frac{1}{\ell_1^2 \ell_2^2 \ell_3^2} \sum\limits_{\stackrel{j=1}{k<l}}^3
\frac{\ell_j^2}{s_j^2} \Bigl[ A_j'(x_j - x_k) - (A_j - A_k) \Bigr] \times \Bigl[ A_j'(x_j - x_l) - (A_j - A_l) \Bigr] \\
&=  \frac{1}{\ell_1^2 \ell_2^2 \ell_3^2} \sum\limits_{\stackrel{j=1}{k<l}}^3
\frac{\ell_j^2}{s_j^2} \Bigl[ \int_{x_j}^{x_k}( A'(x) - A'(x_j)) \, dx  \Bigr] \times \Bigl[ \int_{x_j}^{x_l}( A'(x) - A'(x_j)) \, dx  \Bigr]. 
\end{align*}
 Now the
 hypothesis \eqref{A'-Lipschitz} along with the fact that $(x_k - x_j)^2 \leq (x_k - x_j)^2 + (A_k - A_j)^2 = \ell_l^2$ (and similarly, $(x_l - x_j)^2 \leq \ell_k^2$) lead us to 
\begin{align*}
\frac{1}{2}\, \bigg|\, \mathtt S[\text{Re}(K_\Gamma )](\mathbf z)\bigg| & 
\leq \frac{M^2}{\ell_1^2 \ell_2^2 \ell_3^2}
 \sum_{j=1}^{3} \frac{\ell_j^2}{s_j^2} \frac{(x_k - x_j)^2}{2} \times \frac{(x_{l} - x_j)^2}{2}  \\
&\leq \frac{M^2}{\ell_1^2 \ell_2^2 \ell_3^2}
\sum_{j=1}^{3} \frac{\ell_j^2 \ell_k^2 \ell_l^2}{4\,s_j^2} \\
&= \frac{M^2}{4} \sum_{j=1}^{3} \frac{1}{s_j^2} \\
&\leq \frac34 M^2,
\end{align*}  
where the last inequality follows from the trivial bound $s_j^{2} = 1 + (A'_j)^2 \geq 1$ for all $j$. 
The proof is concluded.
\vskip0.1in
\subsection{Proof of Lemma \ref{L: example-optimal}} Applying Lemma \ref{L:sym-id-curves} to $A(x)= x^3$ and $x_1= -\epsilon\alpha$; $x_2=0$; $x_3=\epsilon\beta$ with $\alpha, \beta\in [0, 1]$ and $\epsilon \in (0, 1)$ we find that 
\vskip0.05in
\begin{align}\label{E:formula}
\frac12\,&\mathtt S[\text{Re}(K_h)](\mathbf z) \\ 
&\hskip0.5in = \frac{\epsilon^2\big(4(\alpha-\beta)^2 + 3\alpha\beta\big) + \epsilon^4X(\alpha, \beta) + \epsilon^6Y(\alpha, \beta) +\epsilon^8 W(\alpha, \beta) + \epsilon^{10} Z(\alpha, \beta)}
{(1+9\epsilon^2\alpha^2)(1+9\epsilon^2\beta^2)(1+\epsilon^4\alpha^4)(1+\epsilon^4\beta^4)
[1+\epsilon^4(\beta^2-\alpha\beta+\alpha^2)^2]}, \nonumber
\end{align}
where each of $X, Y, W$ and $Z$ is a real-valued polynomials in $\mathbb R^2$ and thus achieves a minimum for $(\alpha, \beta)\in [1/2, 1]^2$, which we call $X_0, Y_0, W_0$ and $Z_0$, respectively. Since the Lipschitz constant of $A'(x) = 3x^2$ in the interval $|x|<\epsilon$ is $M_\epsilon = 6\epsilon$, it follows that the numerator in the expression above is bounded below by the quantity
$$
\frac{M^2_\epsilon}{24} + 2\epsilon^4X_0 + 2\epsilon^6Y_0 +2\epsilon^8 W_0 + 2\epsilon^{10}Z_0
$$
for all $(\alpha, \beta)\in [1/2, 1]^2$ and for each $\epsilon>0$.
\vskip0.1in 
\noindent On the other hand, the denominator in the righthand side of \eqref{E:formula} is easily seen to be bounded above by the quantity 
$$
10^2\cdot2^2\cdot\left(1+ \frac{25}{16}\right) = 5^2\cdot 41
$$
for all $(\alpha, \beta)\in [1/2, 1]^2$ and for each $0<\epsilon<1$.
\vskip0.1in
\noindent One now considers various cases, depending on the sign of each of $X_0, Y_0, W_0$ and $Z_0$: if these are all non-negative, then it is clear from the above that 
$$
\mathtt S[\text{Re}(K_h)](\mathbf z)\,\geq c_0 M^2_\epsilon\quad \text{with}\quad c_0 :=
\frac{1}{24\cdot 5^2\cdot 41}\quad \text{and for all}\quad 0<\epsilon<\delta_0:=1.
$$
Suppose next that, say, $X_0<0$ whereas $Y_0, W_0$ and $Z_0$ are all non-negative: in this case it follows that
$$
\mathtt S[\text{Re}(K_h)](\mathbf z)\,\geq \frac{M^2_\epsilon}{24\cdot 5^2\cdot 41}\, -\,
\epsilon^2\frac{2|X_0|}{24\cdot 5^2\cdot 41}\, ,
$$
which gives that
$$
\mathtt S[\text{Re}(K_h)](\mathbf z)\,\geq\, c_0 M^2_\epsilon\quad \text{with}\quad
 c_0 :=
\frac{1}{48\cdot 5^2\cdot 41}\quad \text{and for all}\ \ 0<\epsilon^2<\min\{1, \delta_0^2\},\ \delta_0^2:=\frac{3}{8|X_0|}.
$$
(Here we have used the fact that $M_\epsilon = 6 \epsilon$.) 
\vskip0.1in
\noindent Similarly, the case: $X_0<0, Y_0<0$ and $W_0\geq 0, Z_0\geq 0$ leads to
$$
 c_0 :=
\frac{1}{48\cdot 5^2\cdot 41}\quad \text{and}\quad \delta_0^4:=\frac{3}{8(|X_0|+|Y_0|)}; 
$$
etc. The proof of the Lemma is concluded.
\vskip0.1in
\noindent {\bf Proof of Lemma \ref{L:global-bound-of Menger-for-Lip-A'}.}\quad Let $\mathbf z =(u+iA(u);  x+iA(x); v+iA(v))$ be a three-tuple of distinct points in $\Gamma$. Without loss of generality we may assume that
$$
u<x<v\, .
$$
Lemma \ref{L:menger on curves} gives 
$$
c^2(\mathbf z) = 
\frac{4\,\left[A(x)(v-u) - A(u)(v-x) - A(v)(x-u)\right]^2}
{\ell^2_u\, \ell^2_x\, \ell^2_v}.
$$
Since $\ell_u^2 = (v-x)^2 + (A(v) - A(x))^2$, etc., we have that the denominator in the representation formula for $c^2(\mathbf z)$ is bounded below by the quantity

\begin{equation}\label{E:lower-bd-denom}
\ell^2_u\, \ell^2_x\, \ell^2_v\geq (v-x)^2(x-u)^2(x-u)^2\, .
\end{equation}
\vskip0.1in
\noindent On the other hand,  the numerator in the formula for $c^2(\mathbf z)$ equals
$$
4[(A(x)-A(u))(v-x) -(A(v)-A(x))(x-u)]^2 =
$$
$$
4\left[\int\limits_u^x(A'(t)-A'(u))dt\, (v-x)
 -(A'(v)-A'(u))(x-u)(v-x) 
 -\int\limits_x^v(A'(t)-A'(v))dt\, (x-u)
\right]^2.
$$

But the latter is bounded above by the quantity
$$
4M^2\left[\frac{(x-u)}{2}^{\!2}(v-x) + (v-u)(x-u)(v-x) + \frac{(v-x)}{2}^{\!2}(x-u)\right]^2
$$
and since each of $(x-u)$ and $(v-x)$ is less than $(v-u)$, the quantity above is further bounded by
$$
8 M^2 (v-x)^2(x-u)^2(x-u)^2.
$$
Combining the latter with \eqref{E:lower-bd-denom} we obtain the desired conclusion. 
\vskip0.1in

\subsection{Proof of Theorem \ref{T:positivity-on-curves-with-fixed-curvature}}
 Let us recall from Lemma \ref{L:sym-id-curves} that
\begin{align*}
\mathtt S[\,\mbox{Re}\, K_h\,](\mathbf z) &=\sum\limits_{\stackrel{j}{k<l}}\,
\frac{1}{s^2(x_j)\ell^2_k\,\ell^2_l}\,
\big\{A'(x_j)(x_j-x_k)-\big(A(x_j)-A(x_k)\big)\!\big\} \\ &\hskip1in \times
\big\{A'(x_j)(x_j-x_l)-\big(A(x_j)-A(x_l)\big)\!\big\}\, 
\end{align*}
for any three-tuple $\mathbf z =\{z_1, z_2, z_3\}$ of distinct points on $\Gamma$. 
We claim that each of the three terms in the above summation is non-negative by the assumed fixed concavity of $\Gamma$, that is by the hypothesis that 
\begin{equation}\label{E:concavity}
A''(x)\geq 0 \ \ \text{for every } x\in J\, , \ \ (\text{alt.}\  A''(x)\leq 0\  \text{for every } x\in J).
\end{equation}
To see this,  we assume (without loss of generality) that the three distinct points $\{z_1, z_2, z_3\}$ have been labeled so that
\begin{equation}\label{E:labeled}
x_1<x_2<x_3\, ,\quad \text{where}\ z_j=x_j+iA(x_j).
\end{equation}
Now examining for instance the term corresponding to $j=2$ we find that
\begin{equation}\label{E:integrals}
\big\{A'(x_2)(x_2-x_1)-\big(A(x_2)-A(x_1)\big)\!\big\}
\big\{A'(x_2)(x_2-x_3)-\big(A(x_2)-A(x_3)\big)\!\big\}\ =
\end{equation}
$$
=\ \left(\,\int\limits_{x_1}^{x_2}\!\!\big(A'(x_2)-A'(x)\big)dx\right)\!
\left(\,\int\limits_{x_3}^{x_2}\!\!\big(A'(x_2)-A'(x)\big)dx\right)\ =
$$
$$
=\ \left(\,\int\limits_{x_1}^{x_2}\!\!\big(A'(x_2)-A'(x)\big)dx\right)\!
\left(\,\int\limits_{x_2}^{x_3}\!\!\big(A'(x_2+x_3-t)-A'(x_2)\big)dt\right)\ 
$$
and it is immediate to see that the latter is non-negative because of \eqref{E:concavity} and 
\eqref{E:labeled}. The remaining two terms are dealt with in a similar fashion.
The proof is concluded.

\section{Examples}\label{S:Examples}

\subsection{Failure of global relative boundedness for ${\tt S}[\mathrm{Re}K_\Gamma ]$ and ${\tt S}[\mathrm{Im}K_\Gamma ]$).}\label{SS:failed rel bdd}
Let $A: (-1, 2)\to \mathbb R$ be a smooth function obeying the following constraints:
$$
A(0)=0;\  \ A\left(\frac12\right)=0;\ \ A(1)=0;
$$
$$
A'(0)=0;\ \ \ \ A'\left(\frac12\right)= -1;\ \ A'(1)=0\, .
$$
For instance, the function
\begin{equation}\label{E:rel-bdd-fail}
A(x) = \chi(x)\,\sin2\pi x
\end{equation}
where $\chi$ is in $C^\infty_0\big((0,1)\big)$ and has $\chi (1/2)=1/(2\pi)$, satisfies all of the conditions above.
If we further require that $\chi'(1/2)\neq 0$ then $A''(1/2)\neq 0$ and therefore the curve $\Gamma =\{x+iA(x),\ x\in J = (-1, 2)\}$
 has nonzero curvature $\kappa_0$ at the the point
$$
z_0 := \frac12 \in\Gamma\, .
$$
We claim that for such $\Gamma$ and for $\tilde\epsilon>0$ as in  Corollary \ref{C:bddness}, the interval 
$$
\tilde I = \left(\frac12-\tilde \delta, \frac12 +\tilde \delta\right)
$$
that was obtained there
is strictly contained in $J= (-1, 2)$. To see this, we argue by contradiction and suppose that 
\eqref{E:local-relative-bdd} were to
 hold for any $\mathbf z\in \Gamma^3$. Invoking the conclusions and notation of \cite[Proposition 2.2]{LPg}, it is easy to see that the presumed validity of any of the two inequalities displayed in \eqref{E:local-relative-bdd} is equivalent to the requirement that
$$
\big|1-\mathcal R_h(\mathbf z)\big|\leq \frac{\tilde\epsilon}{\kappa_0^2-\tilde\epsilon}\quad \text{for\ all\ non-collinear\ three-tuples}\ \ \mathbf z\in \Gamma^3.
$$
But the latter would obviously imply that
\begin{equation}\label{E:impossible-h}
|\Rh (\mathbf z)|\leq C \quad \text{for any non-collinear three-tuple}\ \mathbf z\in \Gamma^3
\end{equation}
which is, in fact, not possible. To see this,
consider ordered three-tuples of 
of the form
$$
\mathbf z_\lambda = \big(0;\ \lambda +iA(\lambda); \ 1\big)\in \Gamma^3, \quad 0<\lambda\leq \frac12\, .
$$
Such three-tuples are {\em admissible} in the sense of \cite[Definition 2.1]{LPg}, and the triangles with vertices at $\mathbf z_\lambda$
have the following properties as $\lambda\to1/2$:
$$
\theta_{j, \lambda}\to 0,\ j=1, 2;\ \ \theta_{3,\lambda}\to \pi;\ \ 
\ell_{j,\lambda}\to \frac12,\ j=1, 2; \ \ell_{3,\lambda}=1.
$$
Furthermore, with the notations of \cite[(2.6)]{LPg}
  for each such triangle we have $\alpha_{21}=0$. Thus \cite[Proposition 2.2]{LPg} gives that
$$
\Rh (\mathbf z_\lambda)\ =\ 
$$
$$
=\ \frac{\ell_{1, \lambda}}{4\ell_{2, \lambda}\sin^2\!\theta_{1, \lambda}}
\,\big(\ell_{1, \lambda}\cos (2h(0)-\theta_{1, \lambda}) +
\ell_{2, \lambda}\cos (2h(1)-\theta_{2, \lambda}) -
\ell_{3, \lambda}\cos (2h(\lambda+iA(\lambda)) +\theta_{2,\lambda}-\theta_{1, \lambda})
\big).
$$
Since $A'(0)=A'(1)=0$, 
$$
e^{ih(0)}=e^{ih(1)}=-i=e^{i\pi}\,, \quad \text{thus}\ \ h(0)=h(1)=\pi\, . 
$$
Also
$$
e^{ih\left(\frac12\right)} = \frac{-1-i}{\sqrt{2}} = e^{-i\frac34 \pi}
\,, \quad \text{thus}\ \ h\!\left(\frac12\right) =-\frac34\pi\, .
$$
If condition \eqref{E:impossible-h} were to hold at $\mathbf z_\lambda$ for any $0<\lambda<1/2$
 then it would follow that
$$
\ell_{1,\lambda}\big|\ell_{1, \lambda}\cos (2h(0)-\theta_{1, \lambda}) +
\ell_{2, \lambda}\cos (2h(1)-\theta_{2, \lambda}) -
\ell_{3, \lambda}\cos (2h(\lambda+iA(\lambda)) +\theta_{2,\lambda}-\theta_{1, \lambda})\,
\big|\leq 
$$
$$
\leq\,  C\ell_{2, \lambda}\sin^2\!\theta_{1, \lambda}\quad  \text{for every } 0<\lambda<1/2\, ,
$$
but this is not possible because with our choice of $h$ the lefthand side of this inequality tends to 
$$
\frac12\left(\frac12\cos 2\pi\ +\ \frac12\cos 2\pi\ -\ \cos\left(-\frac32 \pi\right)\right) =\frac12
$$
as $\lambda\to 1/2$, whereas the righthand side tends to 0.
\vskip0.2in
\subsection{Failure of global negativity of ${\tt S}[\mathrm{Im}K_\Gamma ](\mathbf z)$ for $\Gamma$ with fixed concavity.}\label{SS:Im-pos}
Let $A(x)=x^2$ and set $\Gamma =\{x+iA(x),\ x\in\mathbb R\}$.  Consider three-tuples of the form
$$
\mathbf z_\lambda =\big(-\lambda + i\lambda^2\,;\ 0\,;\ \lambda+i\lambda^2\big),\quad \lambda>0.
$$
Lemma \ref{E:RE-graph} gives
$$
\mathtt S\,[\text{Im}\, K_h\,](\mathbf z_\lambda) = \frac{32}{\lambda^2(4+\lambda^2)}
\left(\frac{2+\lambda^2}{8(1+\lambda^2)} - \frac{1}{4+a^2}\right)
$$
and it is clear from the above that $\mathtt S\,[\text{Im}\, K_h\,](\mathbf z_\lambda)>0$ for $\lambda  \gg 1$.
\vskip0.2in

\subsection{Failure of global positivity of \ ${\tt S}[\mathrm{Re}K_\Gamma ](\mathbf z)$ in the absence of fixed concavity of $\Gamma$.}\label{SS: fail-global-pos}

Let $A(x)=x^3$ and set $\Gamma =\{x+iA(x),\ x\in\mathbb R\}$.
       Fix $a>0$ and let $\lambda>0$. Consider three-tuples of the form
$$
\mathbf z_\lambda =\big(\!-a-ia^3\,;\ 0\,;\ \lambda+i\lambda^3\big),\quad \lambda>0.
$$
We claim that
$$
\mathtt S\,[\mathrm{Re}K_h\,](\mathbf z_\lambda)\ <\ 0\quad \text{whenever}\ \lambda \gg a.
$$
To prove the claim we express $\mathtt S\,[\mathrm{Re}K_h\,](\mathbf z_\lambda)$ using \eqref{E:RE-graph} and \eqref{E:integrals}, and obtain
$$
\mathtt S\,[\mathrm{Re}K_h\,](\mathbf z_\lambda)\ =\ I_\lambda \ +\ II_\lambda\ +\ III_\lambda\, \quad \text{where}
$$
$$
\mathtt {I}_\lambda =\ \frac{2\, a\, (2a^2-\lambda^2)}
{(\lambda +a)(1+a^2)\big[1+a^2+\lambda\,(\lambda -a)\big](1+9a^2)}\, ;
$$
\vskip0.05in
$$
\mathtt{II}_\lambda =\ -\,\frac{a\lambda}{ (1+a^2)(1+\lambda^2)}\ <\ 0\quad \text{for any}\ a>0, \lambda>0\,;
$$
\vskip0.05in
$$
\mathtt{III}_\lambda=\ \frac{2\, \lambda\, (2\lambda^2-a^2+a\lambda)}
{(\lambda +a)(1+\lambda^2)\big[1+\lambda^2+a\,(a-\lambda)\big](1+9\lambda^2)}\, .
$$
Note that 
$$
 \mathtt{I}_\lambda <0\quad \text{if}\quad \lambda \gg a\, ,\quad \text{and}\quad  \mathtt{III}_\lambda\ =\ O(\lambda^{-4}).
$$
Thus $\mathtt{I}_\lambda  + \mathtt{II}_\lambda  + \mathtt{III}_\lambda<0$ whenever $\lambda \gg a$. The claim is proved.
\section{Appendix}
\label{S:appendix}
As a point of comparison with the results of this article, it is interesting to note that the boundedness and positivity of $\mathtt{S}[\text{Re} K_h]$ and  $\mathtt{S}[\text{Im} K_h]$ with $K_h$ as in \eqref{E:Kh-def} fail globally on $\mathbb C$ for {\em{any}} continuous, non-constant, globally defined $h:\mathbb C\to \mathbb R$. This appendix summarizes the main results in this direction. The methods of proof bear some similarities to techniques recently developed in  a body of work joint in part by Chousionis, Chunaev, Mateu, Prat and Tolsa \cite{{CMPT}, {CMPT2}, {CP}, {C}, {CMT-1}, {CMT-2}}. 
 For example, 
 we
  use 
  representation formul\ae\,  for symmetrized forms that rely upon a certain labeling scheme of the vertices of a triangle, 
  analogous to
   \cite[Proposition 3.1]{CP} and \cite[Lemma 6]{C},  as well as the computations accompanying the diagrams Figures 3 and 4 in \cite[p. 2738]{C}. 
Hence, for brevity
we limit the exposition here to the statements
 of the pertinent results and defer all proofs to the auxiliary note \cite{LPg}.  
\subsection{Preliminaries}
First note that putting 
$h \equiv 0$ (or a constant) in \eqref{E:Kh-def} yields the original kernel $K_0$ (or a constant multiple of it).  Also it is immediate to see that
 \begin{equation}\label{E:same}
 \mathtt S\,[K_h](\mathbf z) \, =\, 
\mathtt  S[K_0] (\mathbf z) = c^2(\mathbf z)\
 \end{equation}
 for any three-tuple of distinct points
 and for any $h:\mathbb C\to\mathbb R$ (no continuity assumption needed here).
\begin{defn}\label{D:admissible}
We say that an ordered three-tuple of non-collinear points
$(a, b, c)$ is arranged in {\em admissible order} (or is {\em admissible}, for short) if {\tt(i)}
the orthogonal projection of $c$ onto the line determined by $a$ and $b$ falls in the interior of the line segment joining $a$ and $b$, and {\tt(ii)} the triangle with vertices $a, b$ and $c$, henceforth denoted $\Delta (a, b, c)$, has positive counterclockwise orientation.
\end{defn}
\begin{prop}\label{P:symm-ids-Re-Im}
For any non-constant $h:\mathbb C\to \mathbb R$ and for any three-tuple $\mathbf z$
 of non-collinear points in $\mathbb C$ we have
 \item[]\quad 
 \begin{equation}\label{E:new-symm-Re}
    \mathtt S\,[\mathrm{Re} K_h](\mathbf z) 
  \displaystyle{
  =\ c^2(\mathbf z)\left(\frac12\, + \, \Rh (\mathbf z)\right)}
 \end{equation}
\begin{equation}\label{E:new-symm-Im}
  \mathtt S\,[\mathrm{Im} K_h](\mathbf z) 
  \displaystyle{
 =\ c^2(\mathbf z)\left(\frac12\, - \, \Rh (\mathbf z)\right)}
 \end{equation}
 where $\Rh $ is non-constant and invariant under the permutations of the elements of $\mathbf z$. If
 $\mathbf z= (z_1, z_2, z_3)$ is admissible then $\Rh (\mathbf z)$ is represented as follows:

 \begin{equation}\label{E:RH}
\begin{aligned}
\Rh (\mathbf z) &=\frac{2\,\ell_1\ell_2\ell_3}{\left(4\mathrm{Area}\,\Delta(\mathbf z)\right)^2}\, \times 
\Bigl[
\ell_1\cos(\s(z_1) -\tho) + 
\\ 
&\hskip1in + 
\ell_2\cos (\s(z_2) +\ttw)
-\ell_3\cos(\s(z_3) +\ttw -\tho)
\Bigr]\, .
\end{aligned} 
\end{equation} 
Here $\theta_j$ denotes the angle at $z_j$, and $\ell_j$ denotes the length of the side opposite to $z_j$ in $\Delta(\mathbf z)$.
 Also, we have set
 $$\s(z) := 2h(z) -2\ato\, ,\quad z\in\mathbb C,
  $$
where $\ato$ is the principal argument of $z_2-z_1$ (in an arbitrarily fixed coordinate system for $\mathbb R^2$). 
 \end{prop}
\subsection{Failure of universal boundedness and positivity for 
$\mathtt S\,[\mathrm{Re} K_h](\mathbf z)$ and $\mathtt S\,[\mathrm{Im} K_h](\mathbf z)$.}
On account of Proposition \ref{P:symm-ids-Re-Im}, 
the behavior of the symmetrized forms
 of $\mathrm{Re} K_h$ and $\mathrm{Im} K_h$ are reduced to the analysis of the reminder $\Rh$.
 \begin{thm}\label{T:L-infty}
 Suppose that $h:\mathbb C\to \mathbb R$ is continuous.  The following are equivalent:
\begin{itemize}
\item[]
\item[{\tt (i)}]  \ \ 
There is  a constant $C<\infty$, possibly depending on $h$, such that $$|\Rh (\mathbf z)|\leq C$$
\ \ for any  three-tuple $\mathbf z=\{z_1, z_2, z_3\}$ of non-collinear points in $\mathbb C$.

\item[]
\item[{\tt (ii)}] \quad 
$\Rh (\mathbf z) \, =\, 0$ \, for any three-tuple
of non-collinear points in $\mathbb C$.
\item[]
\item[{\tt (iii)}] \quad
 $h$ is constant.
 \item[]
\end{itemize}
\end{thm}
\begin{cor}\label{C:L-infty}
 Suppose that $h:\mathbb C\to \mathbb R$ is continuous. Then
$$
h\ \text{is\ constant}\quad \iff\quad
 \left(\frac12 -\Rh(\mathbf z)\right)\!\! \left(\frac12 +\Rh(\mathbf z)\right) > 0
$$
for any three-tuple of  non-collinear points in $\mathbb C$.
\end{cor}
\noindent In fact more is true.
\begin{thm}\label{T:positive}
 Suppose that $h:\mathbb C\to \mathbb R$ is continuous. 
 
 \vskip0.1in
\begin{itemize}
\item[{\tt (a)}]  \ \  If $h$ is constant, then
$$\displaystyle{\frac12 +\Rh (\mathbf z) > 0}$$
\quad for any  three-tuple
 of non-collinear points.
\item[]
\vskip0.1in
\item[{\tt (b)}] \quad If
 $h$ is not constant, then the function
 $$
 \mathbf z\ \mapsto \frac12 +\Rh (\mathbf z) 
 $$
 \quad
 changes sign.  That is, there exist two three-tuples of non-collinear points 
 
 \ $\mathbf z$ and $\mathbf z'$ such that
 $$
 \frac12 +\Rh (\mathbf z)\,  >\, 0\quad \text{and}\quad  \frac12 +\Rh (\mathbf z')\,  <\, 0.
 $$
\end{itemize}
Furthermore, 
{\tt (a)}
and {\tt (b)} are also true with $\displaystyle{\frac12 - \Rh}$ in place of $\displaystyle{\frac12 + \Rh}$.
\end{thm}
\begin{cor}\label{C:ReKh-positive}
 Suppose that $h:\mathbb C\to \mathbb R$ is continuous. Then
 \begin{itemize}
 \item[{\tt{(a)}}] $h$ is constant \quad $\iff\quad \frac12 +\Rh(\mathbf z) > 0$
 \vskip0.05in
\item[] for all three-tuples
 of  non-collinear points.
\vskip0.1in
\item[{\tt{(b)}}] $h$ is constant \quad $\iff\quad \frac12 -\Rh(\mathbf z) > 0$
\vskip0.05in
\item[] for all three-tuples
 of  non-collinear points.
 \end{itemize}
\end{cor}
\vskip0.2in
\subsection{Further results: the dual kernel of $K_h$}\label{S:last}
\noindent We define the dual kernel of $K_{h} (w, z)$ as 
\begin{equation}  K_h^{\ast} (w, z) = \overline{K_h(z, w)}. \label{def-Kh-star} \end{equation}  
Thus 
\[ K_h^{\ast} (w, z) = \frac{e^{-i h(z)}}{\overline{z} - \overline{w}}. \]
\noindent In particular
$$
   K_0^\ast(w, z) = -
 \overline{K_0(w, z)},
$$
 so that
 $\mathtt S\,[K_0^\ast](\mathbf z)\, =\,  \mathtt S\,[K_0](\mathbf z)\, =\,  c^2(\mathbf z)$.
On the other hand, for non-constant $h$ we have
  \vskip0.05in
\begin{prop}\label{P:symm-id-K-star} 
For any non-constant $h:\mathbb C\to \mathbb R$ and for any three-tuple $\mathbf z$ of non-collinear points
 in $\mathbb C$ we have
\begin{equation}\label{E:sym-adj}
  \mathtt S\,[K_h^\ast](\mathbf z) \ =\, 
c^2(\mathbf z)\, \H(\mathbf z)\,
\end{equation}
where $\H(\mathbf z)$ is a non-constant function of $\mathbf z$ that is 
  invariant under the permutations of the elements of $\mathbf z$. In particular, if $\mathbf z =(z_1, z_2, z_3)$ is admissible then
$\H (\mathbf z)$ has the following
 representation.
\begin{equation} \label{E:H}
\begin{aligned}
\H (\mathbf z) &=\frac{2\,\ell_1\ell_2\ell_3}{\left(4\mathrm{Area}\,\Delta(\mathbf z)\right)^2}\, \times 
\Bigl[
\ell_1\cos(h(z_2) - h(z_3) +\theta_1) + \\ 
&\hskip1in + \ell_2\cos (h(z_1)-h(z_3) - \theta_2)
+\ell_3\cos(h(z_1) -h(z_2) +\theta_3)
\Bigr]\, .
\end{aligned} 
\end{equation} 
Here $\theta_j$ and $\ell_j$ are as in the statement of Proposition \ref{P:symm-ids-Re-Im}. 
\end{prop}

\vskip0.1in

 \begin{thm}\label{T:L-infty-H}
 Suppose that $h:\mathbb C\to \mathbb R$ is continuous.  The following are equivalent:
\begin{itemize}
\item[]
\item[{\tt (i)}]  \ \ 
There is  a constant $C<\infty$, possibly depending on $h$, such that $$|\H (\mathbf z)|\leq C$$
\ \ for any  three-tuple $\mathbf z=\{z_1, z_2, z_3\}$ of non-collinear points in $\mathbb C$.

\item[]
\item[{\tt (ii)}] \quad 
$\H (\mathbf z) \, =\, 1$ \, for any three-tuple
of non-collinear points in $\mathbb C$.
\item[]
\item[{\tt (iii)}] \quad
 $h$ is constant.
 \item[]
\end{itemize}
\end{thm}

\end{document}